\pdfoutput=1

\documentclass{amsart}

\usepackage[bookmarks=true,%
    colorlinks=true,%
    linkcolor=blue,%
    citecolor=blue,%
    filecolor=blue,%
    menucolor=blue,%
    urlcolor=blue,%
    breaklinks=true]{hyperref}

\usepackage{bm}
\usepackage{amsthm}
\usepackage[mathscr]{euscript}
\usepackage{graphicx}

\newcommand{\Q}{\mathbb{Q}}

\newcommand{\SL}{\mathrm{SL}}
\newcommand{\Sp}{\mathrm{Sp}}

\newcommand{\GL}{\mathrm{GL}}

\renewcommand{\qed}{ $\sqcup\!\!\!\!\sqcap$}

\def\-{\overline}

 \def\R{\mathbb{R}}
 \def\Q{\mathbb{Q}}

\def\Spin{{\rm{Spin}}}

\def\mod{\rm{mod}}

\def\qed{ $\sqcup\!\!\!\!\sqcap$}

\def\<{\langle}
\def\>{\rangle}

\newtheorem{theorem}{Theorem}[section]
\newtheorem{lemma}[theorem]{Lemma}
\newtheorem{corollary}[theorem]{Corollary}

\newtheorem{question}[theorem]{Question}

\title{Property FA is not a profinite property}
\author{Tamunonye Cheetham-West}
\author{Alexander Lubotzky}
\author{Alan W. Reid}
\author{Ryan Spitler}

\address{\newline
\newline
Department of Mathematics,\newline
Faculty of Mathematics and Computer Science,\newline
Weizmann Institute of Science,\newline
Rehovot, Israel}
\email{alex.lubotzky@mail.huji.ac.il }
\address{\newline
Department of Mathematics,\newline
Rice University,\newline
Houston, TX 77005, USA.}
\email{tcw@rice.edu,alan.reid@rice.edu,rfs8@rice.edu}

\thanks{The second author was supported by the European Research Council (ERC) under the European Union's Horizon $2020$ research and innovation program (grant agreement No. $882751$), 
and the fourth author was supported by a N.S.F. Postdoctoral Fellowship, DMS-$2103335$.}

\begin{document}

\begin{abstract}
 We exhibit infinitely many pairs of non-isomorphic finitely presented, residually finite groups $\Delta$ and $\Gamma$ with $\Delta$ having Property FA, $\Gamma$ having a non-trivial action on a tree and $\Delta$ and $\Gamma$ having isomorphic profinite completions. 
\end{abstract}



\maketitle

\centerline{\em{Dedicated to Pavel Zalesskii on the occasion of his $60$th birthday}}
%
%
%
%
%

\section{Introduction}
\label{intro}
Let $\Gamma$ be a finitely generated residually finite group, the {\em profinite completion} $\widehat{\Gamma}$ of $\Gamma$
is the inverse limit of the inverse system consisting of the finite quotients $\Gamma/N$ and the natural epimorphisms $\Gamma/N\rightarrow \Gamma/M$. A property $\mathcal{P}$ for $\Gamma$ is said to be a {\em profinite property} if whenever
$\Delta$ is a finitely generated residually finite group with $\widehat{\Delta}\cong\widehat{\Gamma}$, if $\Gamma$ has $\mathcal{P}$ then $\Delta$ also has $\mathcal{P}$. The abelianization is a profinite property (in particular the first Betti number is), a more subtle profinite property is the first $\ell^2$-Betti number \cite[Corollary 3.3]{BCR}.  On the other hand in \cite{Aka}, it is shown that Property T is not a profinite property, and it 
remains an open question of Remeslennikov \cite[Question 5.48]{MK} as to whether being a free group is a profinite property. As is well-known (see \cite{Wat}), if a group $\Gamma$ has Property T then it also has Property FA of Serre; i.e.
every action of $\Gamma$ on a tree has a global fixed point. The main result of this note strengthens that of \cite{Aka}.
\begin{theorem}
\label{main}
There are infinitely many pairs of non-isomorphic, finitely presented, residually finite groups $\Delta$ and $\Gamma$ satisfying the following:
\begin{enumerate}
\item $\widehat{\Delta}\cong\widehat{\Gamma}$;
\item $\Delta$ has Property T and hence Property FA;
\item $\Gamma$ has neither Property T nor Property FA.
\end{enumerate}
\end{theorem}
Recall that a finitely generated group $G$ {\em does not} have Property FA if and only if $G$ is an HNN extension or $G$ admits a decomposition as a non-trivial free product with amalgamation.  For a finitely generated group, being an HNN extension is equivalent to having infinite abelianization, which, as noted above, is a profinite property. Hence a corollary of Theorem \ref{main} is the following.

\begin{corollary}
\label{FPA_NOT}
Being a free product with amalgamation is not a profinite property. \end{corollary}

\noindent To put Theorem \ref{main} and Corollary \ref{FPA_NOT} in context we make two remarks:

\begin{enumerate}
\item The aforementioned examples of Aka \cite{Aka} provide pairs of groups $G$ and $H$ with the same profinite completion with $G$ having Property T,  $H$ does not, but we emphasize that both of $G$ and $H$ has Property FA. This is clear for $G$ since as remarked upon above, having Property T implies Property FA \cite{Wat}. That $H$ has Property FA can be seen as follows: $H$ is an arithmetic lattice in the Lie group $\Spin(n,1)\times \Spin(n,1)$ which has rank $2$, and so $H$ and all subgroups of finite index have finite abelianization. Hence $H$ cannot be an HNN extension. That $H$ does not admit a decomposition as a free product with amalgamation follows from work of Margulis \cite{Mar1} exploiting the fact that $H$
is an arithmetic lattice in a product of rank one groups none of which is p-adic. Alternative proofs of this are now known; see for example \cite{Mo}, \cite{MoS} and \cite{Sh}. In particular the examples of \cite{Aka} cannot be used to prove 
Theorem \ref{main} and Corollary \ref{FPA_NOT}.
\item  It was previously known that there exist non-isomorphic groups that are free products with amalgamation having isomorphic profinite completions. For example, in \cite[Example 4.11]{GZ}, Grunewald and Zalesskii provide examples of non-isomorphic free products with amalgamation where the groups in question are finite, and in \cite[Section 10]{Wil}, examples of non-isomorphic  graph manifold groups are given that are free products with amalgamation associated to non-trivial JSJ decompositions of the 3-manifolds. For emphasis, we note that it remains an open question (see \cite[Question 3]{GJZ}) as to whether being a free product is or is not a profinite property (as remarked upon above,  even for the free group). In connection with this, and in the spirit of this work, in \cite{GJZ}, it was shown that a subgroup of a finitely generated virtually free group $G$ is a free factor if and only if its closure in the profinite completion of $G$ is a profinite free factor
\end{enumerate}
The infinitude of examples stated in Theorem \ref{main} are $S$-arithmetic groups that arise by choice of a rational prime. Fixing a prime $p$, the group $\Delta$ of Theorem \ref{main} is the $S$-arithmetic group $\SL(4,\mathbb{Z}[\frac{1}{p}])$. Hence a corollary of independent interest is the following. Recall that a finitely generated residually finite group
$\Gamma$ is {\em profinitely rigid} if whenever $\Lambda$ is a finitely generated residually finite group with $\widehat{\Lambda}\cong \widehat{\Gamma}$ then $\Lambda \cong \Gamma$.
\begin{corollary}
\label{SL4}
$\SL(4,\mathbb{Z}[\frac{1}{p}])$ is not profinitely rigid.\end{corollary}
It remains open as to whether $\SL(4,\mathbb{Z})$ (or indeed any $\SL(n,\mathbb{Z})$, $n\geq 2$) and also $\SL(n,\mathbb{Z}[\frac{1}{p}])$ for $n=2,3$ are profinitely rigid.

\medskip

\noindent{\bf Acknowledgements:}~{\em  This work was initiated when the second named author visited Rice University to deliver the Chang-Char Tu Distinguished Lecture in Fall 2022, and he wishes to thank the department for its hospitality.
He also wishes to thank the Fields Institute which hosted him whilst this work was carried out. The authors are also very grateful to David Fisher for some useful references, and particularly to Nicolas Monod for correspondence concerning Remark (1) above.}

\section{The groups}
\label{groups}
For the remainder of this note, we fix a rational prime $p$.  As noted in \S \ref{intro}, the group $\Delta$ will be taken to be $\SL(4,\mathbb{Z}[\frac{1}{p}])$, an $S$-arithmetic lattice in $\SL(4,\mathbb{R})\times\SL(4,\mathbb{Q}_p)$.  
We record some facts about $\Delta$.

\begin{lemma}
\label{TnotFA}
The group $\Delta$ is finitely presented, has Property T, and hence has Property FA.\end{lemma}

\begin{proof}
Since $\Delta$ is an $S$-arithmetic lattice it is finitely presented by \cite[Theorem 5.11]{PlatR}.
The groups $\SL(4,\mathbb{R})$ and $\SL(4,\mathbb{Q}_p)$ have Property T (see \cite[Theorem III.5.6]{Mar}), and so
the product $\SL(4,\mathbb{R})\times \SL(4,\mathbb{Q}_p)$ also has Property T by \cite[Corollary III.2.10]{Mar}. Since $\Delta$ is a lattice in $\SL(4,\mathbb{R})\times \SL(4,\mathbb{Q}_p)$, it follows that $\Delta$ has Property T 
by \cite[Theorem III.2.12]{Mar}. A theorem of Watatani \cite{Wat} (see also \cite[Theorem III.3.9]{Mar}), then implies that $\Delta$ has Property FA. \end{proof}

\subsection{Congruence completion and the Congruence Subgroup Property}
\label{SL4CSP}
For convenience we recall the construction of {\em congruence subgroups} of the group $\Delta$.

Let $n\in \mathbb{Z}$ be a positive integer co-prime to $p$, then by reducing entries modulo $n$ we have an epimorphism $\phi_n: \Delta=\SL(4,\mathbb{Z}[\frac{1}{p}])\rightarrow \SL(4,\mathbb{Z}/n\mathbb{Z})$ whose kernel $\Delta(n)$ is the {\em principal congruence subgroup} of 
level $n$. A subgroup $H\leq \SL(4,\mathbb{Z}[\frac{1}{p}])$ is called a {\em congruence subgroup} if there exists an $n$ such that $\Delta(n)\leq H$. It is a result of \cite{BMS} that $\Delta$ has the Congruence Subgroup Property in the sense that
{\em every} subgroup of finite index in $\Delta$ contains some principal congruence subgroup. 

An alternative way to describe this is as follows. For any prime $l$ different from $p$, there is an embedding of $\Q$ into $\Q_l$, the $l$-adic numbers. Using this we obtain an embedding of $\SL(4,\Q)$ into the $S$-adelic group $\prod'_{l\neq p}\SL(4,\mathbb{Q}_l)$, which is the restricted product of the $\SL(4,\mathbb{Q}_l)$ and is given the restricted product topology. The closure of $\Delta$ within $\prod'_{l\neq p}\SL(4,\mathbb{Q}_l)$ is the so-called {\em congruence completion} $\overline{\Delta}$. This is a profinite group which is exactly the inverse limit of the congruence quotients $\SL(4,\mathbb{Z}/n\mathbb{Z})$ of $\Delta$, and in fact by the Strong Approximation Theorem, $\overline{\Delta} \cong \prod_{l \neq p} \SL(4,\mathbb{Z}_l)$, which is a compact, open subgroup of $\prod'_{l\neq p}\SL(4,\mathbb{Q}_l)$. That $\Delta$ has the Congruence Subgroup Property is equivalent to $\widehat{\Delta}\cong \overline{\Delta}$.

\subsection{The second group}
\label{group2}

Before defining the group $\Gamma$ itself, we need to define and establish some facts about certain matrix groups over quaternion algebras. 
We refer the reader to \cite[Chapters 1.4 \& 2.3]{PlatR} for more details on the material described here.

For $K$ a field and $A$ a quaternion algebra over $K$, we can define the algebra $\mathrm{M}(2,A)$ of $2 \times 2$ matrices over $A$. Choosing a quadratic field extension $F/K$ which splits $A$, we can embed $A$ into $\mathrm{M}(2,F)$, 
and hence $\mathrm{M}(2,A)$ embeds into $\mathrm{M}(4,F)$ as $K$-algebras. Restricting the determinant from $\mathrm{M}(4,F)$ to $\mathrm{M}(2,A)$ induces the {\em reduced norm map} $\mathrm{Nrd}: \GL(2,A) \to K^\times$ where $\GL(2,A)$ is the group of invertible elements of $\mathrm{M}(2,A)$. It can be shown that $\mathrm{Nrd}$ does not depend on the choice of $F$ nor the embedding of $A$, and that the image of $\mathrm{Nrd}$ lies in $K^\times$ as stated. We now define 
$$\SL(2,A) = \{ g\in \GL(2,A) :  \mathrm{Nrd}(x)=1\}.$$ 
Note that as described in \cite[Chapter 2.3.1]{PlatR}, if $A$ is a division algebra of quaternions, $\SL(2,A)$ is the group of $K$-rational points of a simple, simply connected algebraic group defined over $K$ of $K$-rank $1$.

Finally we observe that when $A$ is the split quaternion algebra $\mathrm{M}(2,K)$, the field $F$ could be chosen as $K$ itself. So in this case, $\mathrm{M}(2,A) = \mathrm{M}(4,K)$, $\GL(2,A) = \GL(4,K)$, and $\mathrm{Nrd}$ is just the usual determinant of $\GL(4,K)$ so that $\SL(2,A) = \SL(4,K)$. 

To construct the group $\Gamma$ we first specialize the above discussion, and to that end let $A/\mathbb{Q}$ be the definite quaternion algebra ramified at the infinite place and the place associated to our fixed rational prime $p$ (so that $A$ is a division algebra). As noted above, $\SL(2,A)$ is the group of rational points of a simple, simply connected algebraic group over $\Q$, for which the
$\R$-points (resp. the $\Q_p$-points) of this algebraic group are $\SL(2,\mathbb{H})$ (resp. $\SL(2,\mathbb{H}_p)$) where $\mathbb{H}$ is the usual Hamilton quaternions and $\mathbb{H}_p$ is the unique division quaternion algebra over $\Q_p$. Since $A$ is unramified at any prime $l\neq p$ we have $A \otimes_\mathbb{Q} \Q_l \cong \mathrm{M}(2,\Q_l)$, and 
so as described above, the $\Q_l$-points can be identified with $\SL(4,\mathbb{Q}_l)$.

We now consider the $S$-adelic embedding of $\SL(2,A)$ into $\prod'_{l\neq p}\SL(4,\mathbb{Q}_l)$. We will sometimes suppress the embedding and continue to simply refer to $\SL(2,A)$.
Since $\SL(2,\mathbb{H}) \times \SL(2,\mathbb{H}_p)$ is not compact, by the Strong Approximation Theorem, the image of $\SL(2,A)$ under this embedding is dense \cite[Theorem 7.12]{PlatR}. We define the group $\Gamma$ as: 
$$\Gamma = \SL(2,A) \cap \prod_{l \neq p} \SL(4,\mathbb{Z}_l).$$
Hence $\Gamma$ is an $S$-arithmetic subgroup of $\SL(2,A)$ and so is finitely presented by \cite[Theorem 5.11]{PlatR}. Moreover, $\Gamma$ is an irreducible lattice in $\SL(2,\mathbb{H}) \times \SL(2,\mathbb{H}_p)$. Because $\prod_{l \neq p} \SL(4,\mathbb{Z}_l)$ is an open subgroup of $\prod'_{l\neq p}\SL(4,\mathbb{Q}_l)$, in which $\SL(2,A)$ is dense, $\Gamma$ is also dense in $\prod_{l \neq p} \SL(4,\mathbb{Z}_l)$. Since $\Gamma$ is not integral at $p$, the congruence completion of $\Gamma$ is its closure in $\prod'_{l\neq p}\SL(4,\mathbb{Q}_l)$. Thus we have for the congruence completion $\overline{\Gamma} = \prod_{l \neq p} \SL(4,\mathbb{Z}_l)$. Finally, since the $\Q$-rank of
$\SL(2,A)$ is $1$, and the $S$-rank is $2$ with
$S=\{\infty,p\}$,  a result of Raghunathan \cite{Rag} (see also \cite{PR}) applies to show that the $S$-arithmetic group $\Gamma$ also has the Congruence Subgroup Property, so that $\widehat{\Gamma} \cong \overline{\Gamma}$.

\section{Completing the proof of Theorem \ref{main}}
\label{finish}
It is immediate from \S \ref{SL4CSP} and \S \ref{group2} that $\widehat{\Gamma}\cong\widehat{\Delta}$ since both have the Congruence Subgroup Property and $\overline{\Delta}\cong\overline{\Gamma}$.  Thus to complete the proof we must show that
$\Gamma$ does not have Property FA.  We now describe how this is done, and to that end briefly recall the construction of a tree on which $\Gamma$ acts non-trivially (following \cite[Chapter II]{Serre}).

\subsection{The Bruhat-Tits tree $X$ associated to $\SL(2,\mathbb{H}_p)$}
We include a quick description of the tree $X$ associated to $G=\SL(2,\mathbb{H}_p)$. This is similar to the construction of the tree of $\SL_2$ over a local field in \cite[Chapter II]{Serre}, and indeed, it is implicit in \cite[Chapter II]{Serre} (see in particular the discussion on \cite[p. 74-75]{Serre}) that the local field can be replaced by the case at hand, the division algebra $\mathbb{H}_p$.  We briefly comment on the construction of $X$ following \cite[Chapter II]{Serre}.  

The quaternion algebra $\mathbb{H}_p$ has a $p$-adic valuation $\alpha: \mathbb{H}_p^*\rightarrow \mathbb{Z}$ obtained by
composing the reduced norm on $\mathbb{H}_p$ with the $p-$adic valuation on $\mathbb{Q}_p$. In particular, since $\mathbb{H}_p$ is a division algebra, there is a uniformizer for $\mathbb{H}_p$; i.e. an element $\pi\in\mathbb{H}_p$ 
with $\alpha(\pi)=1$. Furthermore, $\mathbb{H}_p$ has a unique maximal order $\mathcal{O}_p=\{x\in\mathbb{H}_p\,|\,\alpha(x)\geq 0\}$ (where we set $\alpha(0)=\infty$). The tree $X$ is constructed using the lattices in the left $\mathbb{H}_p$-vector space $\mathbb{H}_p^2$, that is the finitely generated left $\mathcal{O}_p$ submodules of $\mathbb{H}_p^2$ which generate all of $\mathbb{H}_p^2$ over $\mathbb{H}_p$. The vertices of the tree $X$ are the $\mathbb{H}_p$-homothety classes of lattices, and two classes of lattices are connected by an edge in $X$ exactly when they can be represented by lattices $L$ and $L'$ with $L' < L$ and $L/L' \cong \mathbb{F}_p$ as $\mathcal{O}_p$-modules. This gives rise to a $(p+1)$-regular tree $X$.

\begin{lemma}
\label{action}
The action of $\Gamma$ on the tree $X$ is non-trivial. In particular $\Gamma$ does not have Property FA nor Property T.\end{lemma}
\begin{proof}
The group $\Gamma$ is a dense subgroup of $\SL(2,\mathbb{H}_p)$ which inherits the transitive action of $\SL(2,\mathbb{H}_p)$ on the edges of $X$. Thus, there is no global fixed point for the action of $\Gamma$ on $X$ and therefore, $\Gamma$ does not have Property FA. That $\Gamma$ does not have Property T follows as above using \cite{Wat}.
\end{proof}

\noindent Putting everything from \S \ref{groups} and \S \ref{finish} together completes the proof of Theorem \ref{main}.\qed

\section{Final remarks}\label{final}

Inspired by Corollary \ref{SL4}, one might be tempted to reproduce this construction to find other $S$-arithmetic groups which have the same profinite completions as $\SL(2,\mathbb{Z}[\frac{1}{p}])$ or $\SL(3,\mathbb{Z}[\frac{1}{p}])$. However the analogous construction does not produce the desired results in these cases. For $\SL(2,\mathbb{Z}[\frac{1}{p}])$, one would want to consider an $S$-arithmetic subgroup of $A^1$, the group of norm $1$ elements in the quaternion algebra $A$ which we  defined previously. Unfortunately, such an $S$-arithmetic group would be a lattice in the compact group $\mathbb{H}^1 \times \mathbb{H}_p^1$, and hence would be finite.  

In fact, this is
the only other case that needs to be considered when searching for an $S$-arithmetic group with the same profinite
completion as $\SL(2,\mathbb{Z}[\frac{1}{p}])$. Because $\SL(2,\mathbb{Z}[\frac{1}{p}])$ has the Congruence Subgroup
Property, any $S$-arithmetic subgroup with the same profinite completion would have to have the same congruence
completion as well. If it did not, Strong Approximation could be used to provide a finite quotient that
$\SL(2,\mathbb{Z}[\frac{1}{p}])$ does not have. One can then check that $A^1$ is the only other potential
$\Q$-group whose local behavior is consistent with this congruence completion. 

A similar situation occurs in the case of $\SL(3,\mathbb{Z}[\frac{1}{p}])$, which excludes the possibility of a different $S$-arithmetic group having the same profinite completion. Using the ``road map" set out in \cite{BMRS}, one can show that to establish whether or not these groups are profinitely rigid
actually reduces to answering the following question.

\begin{question}
If $\Gamma$ is $\SL(2,\mathbb{Z}[\frac{1}{p}])$ or $\SL(3,\mathbb{Z}[\frac{1}{p}])$, can $\Gamma$ contain a finitely generated proper subgroup $H$ whose inclusion induces an isomorphism $\widehat{H} \cong \widehat{\Gamma}$?
\end{question}

If such an $H$ existed, $(\Gamma,H)$ would be a Grothendieck pair (see \cite{BG}, for example). The case of $\Gamma = \SL(2,\mathbb{Z}[\frac{1}{p}])$ is especially provocative because it splits as an amalgam, which would have significant consequences for the structure of the potential subgroup $H$.

\end{document}